\documentclass[psamsfonts]{amsart}
\usepackage[utf8]{inputenc}
\usepackage{amsfonts}
\usepackage[hidelinks]{hyperref}
\hypersetup{
pdftitle={Notes on Forcing and its Applications},
pdfsubject={Mathematics, Set Theory, Forcing},
pdfauthor={Luciano Salvetti, Tonatiuh Matos-Wiederhold},
pdfkeywords={}
}
\usepackage{amsmath}
\usepackage{xcolor}
\usepackage{amsthm}
\usepackage{pdflscape}
\usepackage{mathrsfs}

\newtheorem{thm}{Theorem}[section]
\newtheorem{cor}[thm]{Corollary}
\newtheorem{prop}[thm]{Proposition}
\newtheorem{lem}[thm]{Lemma}

\newtheorem*{claim}{Claim}

\theoremstyle{definition}
\newtheorem{defn}[thm]{Definition}
\newtheorem{question}[thm]{Question}

\theoremstyle{remark}

\makeatletter
\let\c@equation\c@thm
\makeatother
\numberwithin{equation}{section}

\title{Uncountable sets and an infinite linear order game}

\author{Tonatiuh Matos-Wiederhold \qquad Luciano Salvetti}

\begin{document}

\maketitle
{\centering\tiny\vspace{-0.6cm}Dept.\ of Mathematics\\ University of Toronto\\\texttt{\{t.wiederhold,luciano.salvetti\}@mail.utoronto.ca}\\}

\begin{abstract} 
    An infinite game on the set of real numbers appeared in Matthew Baker's work [Math. Mag. 80 (2007), no. 5, pp. 377--380] in which he asks whether it can help characterize countable subsets of the reals. This question is in a similar spirit to how the Banach-Mazur Game characterizes meager sets in an arbitrary topological space.
    
    In a recent paper, Will Brian and Steven Clontz prove that in Baker's game, Player II has a winning strategy if and only if the payoff set is countable. They also asked if it is possible, in general linear orders, for Player II to have a winning strategy on some uncountable set.
    
    To this we give a positive answer and moreover construct, for every infinite cardinal $\kappa$, a dense linear order of size $\kappa$ on which Player II has a winning strategy on \emph{all} payoff sets. We finish with some future research questions, further underlining the difficulty in generalizing the characterization of Brian and Clontz to linear orders.
\end{abstract}

\maketitle

\section{Introduction}

  In \cite{baker}, Matt Baker introduces the following game, called the \textit{Cantor Game}, on the real numbers: Fix a subset $S\subseteq\mathbb{R}$. Player I starts by picking a real number $a_0$. Then, Player II picks a real number $b_0$ such that $a_0<b_0$. In the $n$-th turn, Player I picks a real number $a_n$ such that $a_{n-1}<a_n<b_{n-1}$ and then Player II picks a real number $b_n$ such that $a_n<b_n<b_{n-1}$. After $\omega$-many turns they form two sequences of real numbers $\{a_n\}$ and $\{b_n\}$ such that $a_0<a_1<\cdots<b_1<b_0$. We say that Player I wins if and only if there exists $x\in S$ such that $a_n<x<b_n$ for all $n<\omega$. Otherwise, we say that Player II wins. It is clear that such game can be defined on more general linear orders. Our Proposition~\ref{prop:arrow} generalizes the following observation when playing the game in any dense linear order.
  
  \begin{prop}[Baker, 2006]\label{prop:countable win}
      If $S$ is countable, then Player II has a winning strategy in the Cantor Game.
  \end{prop}

  The Cantor Game was introduced as a way of proving that $\mathbb{R}$ is uncountable, which follows by the previous observation. Baker also proved that if $S$ contains a perfect set, then Player I has a winning strategy. Later in \cite{cantor},  M.D. Ladue (a student of Baker) proved that Player I has a winning strategy if and only if the set $S$ contains a perfect set. The question of whether the converse of proposition \ref{prop:countable win} is true remained open until 2022. In \cite{brian-clontz}, W. Brian and S. Clontz used elementary submodels to prove the converse statement of proposition \ref{prop:countable win} above. They also motivate our paper's exploration of the general setting studied in Section~\ref{sec:linear}. In particular, they asked the following:
  
  \begin{question}[Brian-Clontz, 2023]\label{question:brian-clontz}
    Is there a linear order $(X,<)$ and an uncountable $S\subseteq X$ on which Player II has a winning strategy?
  \end{question}

  In this paper, we provide an affirmative answer to this question and study this game in the general setting.

  \textbf{Acknowledgements.} We thank Stevo Todorčević for his valuable insights and guidance regarding this work. We also thank Will Brian for reading our draft and for his helpful advice.

\section{Preliminaries}\label{sec:prelim}

Given an ordinal $\alpha$, we denote by $\alpha^*$ the converse of $\alpha$, i.e., the conversely well-ordered set $(\alpha,>)$.

\begin{defn}
    Given a linearly ordered set $(X,<)$, we use the symbol $X\not\rightarrow (\omega^*)_\omega^1$ to abbreviate the fact that $X$ is a countable union of well-ordered sets. This definition appears in a more general context in \cite{partition}.
\end{defn}

 Recall that a linearly ordered set is scattered if it does not contain copies of the rationals. A linearly ordered set is said to be $\sigma$-scattered if it is a countable union of scattered sets. Given an order type $\gamma$, we say that a linearly ordered set is $\gamma$-free if it does not contain copies of $\gamma$. The following result can be found in \cite{partition}. We include the proof for the reader's convenience.

\begin{thm}[Todorčević, 1985]\label{thm:todorcevic}
    Let $X$ be a linearly ordered set. The following are equivalent:
    \begin{itemize}
        \item [(i)] $X\not\rightarrow (\omega^*)_\omega^1$
        \item [(ii)] $X$ is $\sigma$-scattered and $\omega_1^*$-free.
    \end{itemize}
\end{thm}

\begin{proof}
    $(i)\Rightarrow(ii)$ is clear. We show $(ii)\Rightarrow(i)$. It suffices to assume that $X$ is scattered (and $\omega_1^*$-free). We define the following equivalence relation on $X$: $x\sim y$ if and only if the closed interval $[x,y]\not\rightarrow (\omega^*)_\omega^1$. Notice that every equivalence class $\mathcal{C}$ is convex, i.e., if $x,y\in\mathcal{C}$ and $x<z<y$, then $z\in\mathcal{C}$. Let $[x]$ be the equivalence class of $x$. Since every equivalence class is convex, the order $[x]<[y]$ if and only if $x<y$ is well-defined. Hence, we can view the quotient $X/\!\sim$ as a subspace of $X$. Thus, $X/\!\sim$ is also scattered by assumption.
    \begin{claim}
        $[x]_{\geq}:=\{y\in[x]:y\geq x\}\not\rightarrow(\omega^*)_\omega^1$ for all $x$.
    \end{claim}
    \begin{proof}
        Let $\{y_\alpha:\alpha<\kappa\}$ be a cofinal subset of $[x]_{\geq}$ with $y_0=x$. So: $$[x]_{\geq}=\bigcup_{\alpha<\kappa}[y_\alpha, y_{\alpha+1}]$$
        Since $[y_\alpha, y_{\alpha+1}]\not\rightarrow(\omega^*)_\omega^1$, then write $[y_\alpha, y_{\alpha+1}]=\bigcup_{n<\omega}A_n^\alpha$ where $A_n^\alpha$ is $\omega^*$-free. Let $A_n:=\bigcup_{\alpha<\kappa}A_n^\alpha$. Notice that $A_n$ is $\omega^*$-free since $\kappa$ is $\omega^*$-free. Then, $[x]_\geq=\bigcup_{n<\omega}A_n$, so $[x]_\geq\not\rightarrow(\omega^*)_\omega^1$.
    \end{proof}
    \begin{claim}
        $[x]_{\leq}:=\{y\in[x]:y\leq x\}\not\rightarrow(\omega^*)_\omega^1$ for all $x$.
    \end{claim}
    \begin{proof}
        Let $\{y_\alpha:\alpha<\kappa\}$ be a coinitial subset of $[x]_\leq$ with $y_0=x$. Since $X$ is $\omega_1^*$-free, then $\kappa$ is countable so:
        $$[x]_\leq=\bigcup_{n<\omega}[y_{n+1},y_n]$$
        Since $[y_{n+1},y_n]\not\rightarrow(\omega^*)_\omega^1$ for all $n<\omega$, then $[x]_{\leq}\not\rightarrow(\omega^*)_\omega^1$.
    \end{proof}
    The previous two claims imply that $[x]\not\rightarrow(\omega^*)_\omega^1$ for all $x\in X$. Now we show that actually $X=[x]$ for some $x$.
    \begin{claim}
        $(X/\!\sim,<)$ is dense, i.e., if $[x]<[y]$ then there is $z\in Z$ such that $[x]<[z]<[y]$.
    \end{claim}
    \begin{proof}
        Let $[x]<[y]$ and suppose there is no $z$ such that $[x]<[z]<[y]$. Then, $[x,y]\subseteq [x]_\geq\cup[y]_\leq\not\rightarrow(\omega^*)_\omega^1$ which implies $[x]=[y]$.
    \end{proof}
    This last claim shows that $X$ must have one equivalence class, since otherwise $X/\!\sim$ would contain a copy of the rationals.
\end{proof}

It is clear that a subset $S$ of $\mathbb{R}$ is countable if and only if $S\not\rightarrow (\omega^*)_\omega^1$. The following generalizes this fact.

\begin{cor}\label{cor:sigma-countable}
    If $X$ is a $\omega_1$-free, $\omega_1^*$-free linearly ordered set. Then, the following are equivalent for any $S\subseteq X$:
    \begin{enumerate}
        \item $S$ is countable.
        \item $S\not\rightarrow (\omega^*)_\omega^1$.
        \item $S$ is $\sigma$-scattered.
    \end{enumerate}
\end{cor}

\section{The Baker Game on linear orders}\label{sec:linear}
  We may define a variation of Baker's Game in any dense linear order $(X,<)$. Given $S\subseteq X$, we denote by $BG(X,S)$ the game where two players take turns selecting elements $a_n,b_n$ of $X$ as in the Cantor Game for the reals. Player I wins if after $\omega$-many turns there is an $x\in S$ such that for all $n<\omega$, $a_n<x<b_n$, and Player II wins otherwise. The density condition ensures the game never reaches a stalemate. For the rest of the section, $X$ shall be a linearly ordered set and $S$ a subset of $X$. Let us first improve Proposition~\ref{prop:countable win}:

\begin{prop}\label{prop:arrow}
    If $S\not\rightarrow(\omega^*)_\omega^1$, then Player II has a winning strategy for $BG(X,S)$.

    In particular, every countable $S$ has this property.
\end{prop}

  \begin{proof}
    Let $S=\bigcup_{n<\omega}S_n$ where each $S_n$ is $\omega^*$-free (that is, well-founded). In the $n$-th turn, Player II plays $b_n=\min(S_n\cap(a_n,b_{n-1}))$ if $S_n\cap(a_n,b_{n-1})\neq\emptyset$ and any legal move otherwise (by density, there must always be one). Now, suppose that for all $n<\omega$, $a_n<x<b_n$. This implies that $x\notin S_n$ for all $n<\omega$, and hence, $x\notin S$. In other words, we described a winning strategy for Player II.
  \end{proof}

    In general $S\not\rightarrow(\omega^*)_\omega^1$ does not imply that $S$ is countable (unless it contains no copies of $\omega_1$ nor $\omega_1^*$). The previous proposition implies that if $S$ is an (uncountable) well-ordered set, then Player II has a winning strategy for $BG(X,S)$. We now show that if $S$ is conversely well-ordered, then Player II has a winning strategy for $BG(X,S)$. Notice that a conversely well-ordered set $S$ has the property that $S\rightarrow(\omega^*)_\omega^1$, so our result will show that, in general, the converse of Proposition ~\ref{prop:arrow} is false. This will follow from the next theorem.

\begin{thm}\label{thm:blocks}
    Suppose that $\gamma$ is an ordinal and $\{S_\alpha:\alpha<\gamma\}$ is a collection of subsets of $X$ on which Player II has a winning strategy. Furthermore, assume that $S_\beta<S_\alpha$ for all $\alpha<\beta<\gamma$ (i.e.\, $x<y$ for all $x\in S_\beta$ and $y\in S_\alpha$). Then, Player II has a winning strategy for $BG\left(X,\bigcup_{\alpha<\gamma}S_\alpha\right)$.
\end{thm}

\begin{proof}
    Player I starts by playing $a_0\in X$. If there is $b_0>a_0$ such that $b_0\leq y$ for all $y\in\bigcup_{\alpha<\gamma}S_\alpha$, then Player II wins by picking $b_0$. Otherwise, Player II picks any legal move in the first turn. In this case, $a_1>a_0$, so let $\alpha_1<\gamma$ be the smallest such that $y_1<a_1$ for some $y_1\in S_{\alpha_1}$. If there is a legal move $b_1>a_1$ such that $b_1\leq y$ for all $y\in\bigcup_{\alpha<\alpha_1}S_\alpha$, then Player II picks such $b_1$ and then uses his strategy for $S_{\alpha_1}$. Otherwise, Player II plays any legal move $b_1$ and then, since $a_2>a_1$, pick the smallest $\alpha_2<\alpha_1$ such that $y_{2}<a_2$ for some $y_2\in S_{\alpha_2}$.
    
    By repeating this process, since $\gamma$ is well-ordered, we cannot have a decreasing sequence $\dots<\alpha_2<\alpha_1<\gamma$. Hence, at some turn $N$, there is a legal move $b_N>a_N$ for Player II such that $b_N\leq y$ for all $y\in\bigcup_{\alpha<\alpha_N}S_\alpha$. After that, Player II continues playing with his strategy for $BG(X,S_{\alpha_N})$.
    
    We just described a winning strategy for Player II. Indeed, given any $x\in X$ such that $a_n<x<b_n$ for all $n<\omega$, we have the following: given that $y_N<a_N<x$, we have that $x\notin S_{\alpha}$ for all $\alpha>\alpha_N$; also $x\notin S_{\alpha}$ for all $\alpha<\alpha_N$ by the way $b_N$ was played. Finally, $x\notin S_{\alpha_N}$, since Player II used his winning strategy for $BG(X,S_{\alpha_N})$.
\end{proof}

\begin{cor}
    If $S$ is conversely well-ordered, then Player II has a winning strategy on $BG(X,S)$.
\end{cor}

\begin{cor}\label{cor:any_kappa}
    For any infinite cardinal $\kappa$, there is a dense linear order $(X,<)$ of size $\kappa$ on which Player II has a winning strategy on every subset of $X$.
\end{cor}

\begin{proof}
    We take $X=\kappa^*\times\mathbb Q$ ordered lexicographically. Concretely, given $x=(\alpha,p)$ and $y=(\beta, q)$ elements of $X$, $x\leq y$ if either $\beta<\alpha$ or $\alpha=\beta$ and $p<_\mathbb Qq$. Since the rationals are countable and $\kappa$ is infinite, clearly $X$ has the desired size. The fact that $X$ is dense follows from the fact that $\mathbb Q$ is dense and unbounded.

    Now let $S\subseteq X$ and define, for each $\alpha<\kappa$, $S_\alpha=\{\alpha\}\times\mathbb Q$. Clearly, each $S_\alpha$ is countable, and thus Player II has a winning strategy for $BG(X,S_\alpha)$, by virtue of Proposition~\ref{prop:arrow}. Since $S=\bigcup_{\alpha<\kappa}S_\alpha$ and the $S_\alpha$ are ordered the way that Theorem~\ref{thm:blocks} requires, the result follows.
\end{proof}

\section{Concluding remarks and open questions}

We first turn to the question: \textit{when does the converse of Proposition \ref{prop:arrow} hold?}. By the previous results such orderings must be $\omega_1$-free and $\omega_1^*$-free. Separable and even ccc linear orderings have this property. The separable case was established in \cite{brian-clontz}. Then, the next step is to study dense linear orders that are $\omega_1$-free, $\omega_1^*$-free but not separable. The typical example of such ordering is the Aronszajn ordering.

\begin{defn}
    An Aronszajn line (A-line) is an uncountable linearly ordered set that is $\omega_1$-free, $\omega_1^*$-free and contains no uncountable subset of the reals.
\end{defn}

Following \cite{A-orderings} we define a proper decomposition for A-lines. Let $A$ be a dense Aronszajn line. Let $D_0:=\emptyset$. Given a countable $D_\alpha\subseteq A$, consider the equivalence relation on $A\backslash D_\alpha$: $x\sim_\alpha y$ if and only if there is no $z\in D_\alpha$ such that $x<z<y$. Notice that each $\sim_\alpha$ class is convex. Let $T_\alpha$ be the set of all $\sim_\alpha$ classes and let $D_\alpha'$ be a selector of the $\sim_\alpha$ classes. Since $D_\alpha$ is countable and dense in $D_{\alpha+1}:=D_\alpha\cup D_\alpha'$, then $D_{\alpha+1}$ must be countable because $A$ does not contain uncountable separable subsets. Given $D_\alpha$ for $\alpha<\gamma$ and $\gamma$ limit, we let $D_\gamma=\bigcup_{\alpha<\gamma}D_\alpha$. Since $A$ does not contain copies of $\omega_1$ and $\omega_1^*$, we have that $A=\bigcup_{\alpha<\omega_1}D_\alpha$. Also,  $(T,\supseteq)$ is an Aronszajn tree where $T=\bigcup_{\alpha<\omega_1}T_\alpha$.

\begin{question}
    Is there a dense Aronszajn line $A$ where Player II has a winning strategy in $BG(A,S)$ for all $S\subseteq A$?
\end{question}

Another class of lines where a similar decomposition can be performed is the class of nowhere separable Souslin continua:

\begin{defn}
    A Souslin continuum is a non-separable dense, complete linearly ordered set that has the countable chain condition (ccc), i.e., every family of pairwise disjoint open intervals is countable. A Souslin continuum is nowhere separable if no non-empty open interval is separable.
\end{defn}

We define a proper decomposition for the nowhere separable Souslin continuum $L$ as follows: let $D_0:=D\neq\emptyset$ be any countable set. Given a countable $D_\alpha\subseteq L$ consider the equivalence relation on $L\backslash\overline{D_\alpha}$: $x\sim_\alpha y$ if and only if there is no $z\in D_\alpha$ such that $x<z<y$. Notice that each $\sim_\alpha$ class is convex. Let $T_\alpha$ be the set of all $\sim_\alpha$ classes and let $D_\alpha'$ be a selector of the $\sim_\alpha$ classes. Since $L$ has the ccc, then $T_\alpha$ is countable. Let $D_{\alpha+1}=D_\alpha\cup D_\alpha'$. For $\gamma<\omega_1$ limit let $D_\gamma=\bigcup_{\alpha<\gamma}D_\alpha$.

\begin{lem}
    If $\{D_\alpha:\alpha<\omega_1\}$ is a proper decomposition of a nowhere separable Souslin continuum $L$, then $$L=\bigcup_{\alpha<\omega_1}\overline{D_\alpha}$$
\end{lem}

\begin{proof}
    Suppose that $x\in L$ and $x\notin \bigcup_{\alpha<\omega_1}\overline{D_\alpha}$. For each $\xi<\omega_1$ let $I_\xi^0:=[x]_\xi$ (the $\sim_\xi$ class of $x$). Note that $I_\eta^0\subseteq I_\xi^0$ for all $\xi<\eta$.
    \begin{claim}
        For all $\alpha<\omega_1$, $|T_\alpha|>1$.
    \end{claim}
    \begin{proof}
        Let $x_0\in D_\alpha$. Since $L$ is nowhere separable, $\overline{D_\alpha}$ is nowhere-dense so pick $y_1\in (L\backslash\overline{D_\alpha})\cap (-\infty, x_0)$ and $y_1\in (L\backslash\overline{D_\alpha})\cap (x_0,\infty)$. Then, $[y_1]_\alpha,[y_2]_\alpha\in T_\alpha$.
    \end{proof}
    For each $\xi<\omega_1$ let $I_{\xi+1}^1\in T_{\xi+1}$ such that $I_{\xi+1}^1\cap I_{\xi+1}^0=\emptyset$. Then, $\{I_{\xi+1}^1:\xi<\omega_1\}$ is an uncountable pairwise disjoint family of open sets, contradicting ccc.
\end{proof}

\begin{question}
    Can this proper decomposition be used to show that in a nowhere separable Souslin continuum $L$, Player II has a winning strategy for $BG(L,S)$ if and only if $S$ is countable?
\end{question}


\bibliographystyle{alpha}
\nocite{*}
\bibliography{baker_game}

\end{document}